\documentclass[11pt,reqno]{article}

\usepackage[usenames]{color}
\usepackage{graphicx}
\usepackage{amscd}
\usepackage[english]{babel}
\usepackage{color}
\usepackage{fullpage}
\usepackage{psfig}
\usepackage{graphics,amsmath,amssymb}
\usepackage{amsthm}
\usepackage{amsfonts}
\usepackage{latexsym}
\usepackage{epsf}
\usepackage{float}

\usepackage[colorlinks=true,
linkcolor=webgreen,
filecolor=webbrown,
citecolor=webgreen]{hyperref}

\definecolor{webgreen}{rgb}{0,.5,0}
\definecolor{webbrown}{rgb}{.6,0,0}

\newcommand{\seqnum}[1]{\href{http://oeis.org/#1}{\underline{#1}}}
\newcommand{\beql}[1]{\begin{equation}\label{#1}}
\newcommand{\eeq}{\end{equation}}
\newcommand{\eqn}[1]{(\ref{#1})}

\newcommand{\sP}{{\cal{P}}}

\newtheorem{thm}{Theorem}{\bfseries}{\itshape}
{\bfseries}{\itshape}
{\bfseries}{\itshape}
{\bfseries}{\itshape}
{\bfseries}{\itshape}

\newcommand{\al}{\alpha}

\newcommand{\ka}{\kappa}

\newcommand{\la}{\lambda}
\newcommand{\si}{\sigma}

\begin{document}
\theoremstyle{plain}

\begin{center}
{\large\bf The Yellowstone Permutation } \\
\vspace*{+.2in}

David L. Applegate \\
AT\&T \\
One AT\&T Way \\
Bedminster, NJ 07921 \\
USA \\
\href{mailto:david@research.att.com}{\tt david@research.att.com} \\ 
\ \\
Hans Havermann \\
11 Sykes Ave. \\
Weston, ON M9N 1C8 \\
Canada \\
\href{mailto:gladhobo@teksavvy.com}{\tt gladhobo@teksavvy.com} \\
\ \\
Robert G. Selcoe \\
16214 Madewood St. \\
Cypress, TX  77429 \\
USA \\
\href{mailto:rselcoe@entouchonline.net}{\tt rselcoe@entouchonline.net} \\ 
\ \\
Vladimir Shevelev \\
Department of Mathematics \\
Ben-Gurion University of the Negev \\
Beer-Sheva 84105\\
Israel \\
\href{mailto:shevelev@bgu.ac.il}{\tt shevelev@bgu.ac.il} \\ 
\ \\
N. J. A. Sloane\footnote{To whom correspondence should be addressed.} \\
The OEIS Foundation Inc. \\
11 South Adelaide Ave. \\
Highland Park, NJ 08904 \\
 USA \\
 \href{mailto:njasloane@gmail.com}{\tt njasloane@gmail.com} \\
 \ \\ 
Reinhard Zumkeller \\
Isabellastrasse 13 \\
D-80798 Munich \\
Germany \\
\href{mailto:reinhard.zumkeller@gmail.com}{\tt reinhard.zumkeller@gmail.com}

\vspace*{+.1in}

March 7, 2015
\vspace*{+.1in}

{\bf Abstract}
\end{center}

Define a sequence of positive integers by
the rule that $a(n)=n$ for $1 \le n \le 3$, and for $n\ge4$, $a(n)$
is the smallest number not already in the sequence
which has a common factor with $a(n-2)$ and is relatively
prime to $a(n-1)$.
We show that this is a permutation of the positive integers.
The remarkable graph of this sequence consists of runs
of alternating even and odd numbers, interrupted by 
small downward spikes followed by large upward spikes,
suggesting the eruption of geysers in Yellowstone National Park.
On a larger scale the points appear to lie on infinitely many
distinct curves.
There are several unanswered questions  concerning the locations
of these spikes and the equations for these curves.

\section{Introduction}\label{Sec1}
Let $(a(n))_{ n \ge 1}$ be defined
as in the Abstract.  This is sequence \seqnum{A098550}\footnote{Throughout this article, six-digit numbers prefixed by A 
refer to entries in \cite{OEIS}.} 
in the {\em On-Line Encyclopedia of Integer Sequences} 
\cite{OEIS}, contributed by Zumkeller in 2004.
Figures~\ref{Fig1} and \ref{Fig2} show two different views of its graph,
and the first 300 terms are given in Table \ref{Tab1}.
Figure~\ref{Fig1} shows terms $a(101)=47$ through $a(200)=279$, with
successive points joined by lines.
The downward spikes occur when $a(n)$ is a prime, and the larger upward spikes
(the ``geysers'', which suggested our
name for this sequence) happen two steps later.
In the intervals between spikes the sequence
alternates between even and odd values in a fairly smooth way.

Figure~\ref{Fig2} shows the first 300,000 terms, without lines connecting the points.
On this scale the points appear to fall on or close to a number 
of quite distinct curves.  The primes lie on the lowest curve (labeled ``p''),
and the even terms on the next curve (``E'').
The red line is the straight line $f(x)=x$, included for reference (it is not
part of the graph of the sequence).
The heaviest curve (labeled ``C''), just above the red line, consists of almost all the odd
composite numbers. The higher curves are the relatively sparse ``$\ka$p'' points, 
to be discussed in Section \ref{Sec3} when we study the growth
of the sequence more closely. It seems very likely that there are infinitely
many curves in the graph, although this, like other properties to be mentioned in 
Section \ref{Sec3}, is at present only a conjecture.
We are able to show that every number appears in the sequence,
and so $(a(n))_{n \ge 1}$ is a permutation of the natural numbers.

The definition of this sequence
resembles that of the EKG sequence (\seqnum{A064413}, \cite{EKG}),
which is that $b(n)=n$ for $n=1$ and $2$, and for $n \ge 3$,
$b(n)$ is the smallest number not already in the sequence
which has a common factor with $b(n-1)$.  However, 
the present sequence seems considerably more complex.
(The points of the EKG sequence fall on or near just three curves.)
Many other permutations of the natural numbers
are discussed in \cite{JHC1972, JHC2013, Erdos, UPINT}.

\begin{figure}[!ht]
\centerline{\includegraphics[angle=0, width=4in]{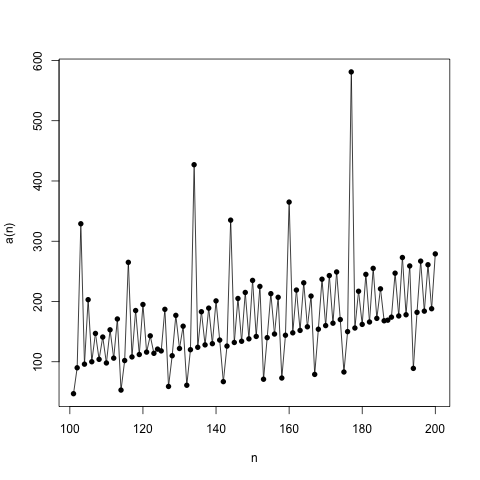}}
\caption{Plot of terms $a(101)$ through $a(200)$ of the Yellowstone permutation.
The downward spikes occur when $a(n)$ is a prime, and the larger upward spikes
(the ``geysers'') happen two steps later.
}
\label{Fig1}
\end{figure}

\begin{figure}[!ht]
\centerline{\includegraphics[angle=0, width=6.75in]{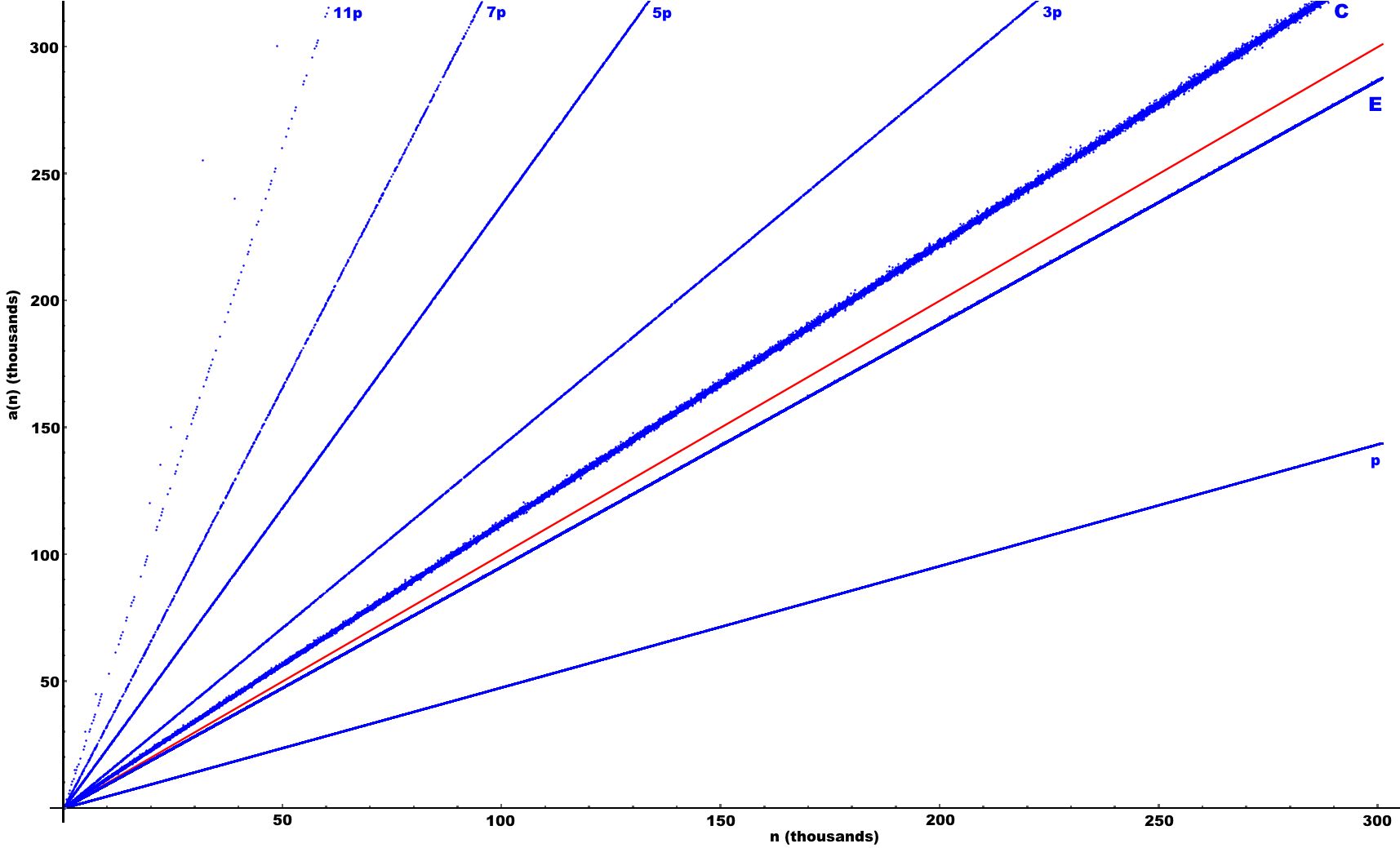}}
\caption{Scatterplot of the first 300,000 terms. The primes lie
on the lowest line (labeled ``p''), the even numbers on the second line (``E''), 
the majority of the odd
composite numbers on the third line (``C''), and the ``$\ka$p''
points on all the higher lines. The lines are not actually straight,
except for the red line $f(x)=x$, which
is included for reference.}
\label{Fig2}
\end{figure}

\section{Every number appears}\label{Sec2}

\begin{thm}\label{Th1}
$(a(n))_{n \ge 1}$ is a permutation of the natural numbers.
\end{thm}
\begin{proof}
By definition, there are no repeated terms, so we need only 
show that every number appears. There are several steps in the argument.

(i) The sequence is certainly infinite, because the term 
$p a(n-2)$ is always  a candidate for $a(n)$, 
where $p$ is a prime not dividing any of $a(1), \ldots, a(n-1)$.

(ii) The set of primes that divide terms of the sequence is
infinite.  For if not, there is a prime $p$ such that every
term is the product of primes $< p$.  Using (i), let $m$ be
large enough that $a(m) > p^2$, and let $q$ be a common prime factor
of $a(m-2)$ and $a(m)$.  Since $q < p$, $qp < p^2 < a(m)$, and so $qp$
would have been a smaller choice for $a(m)$, a contradiction.

(iii) For any prime $p$, there is a term divisible by $p$.  For
suppose not. Then no prime $q>p$ can divide any term, for if it did,
let $a(n) = tq$ be the first multiple of $q$ to appear.  But then we
could have used $tp < tq$ instead.  So every prime divisor is $< p$,
contradicting (ii).

(iv) Any prime $p$ divides infinitely many terms. For suppose not.
Let $N_0$ be such that $p$ does not divide $a(n)$ for $n \ge N_0$.
Choose $i$ large enough that $p^i$ does not divide any term in the sequence,
and choose a prime $q > p^i$ which does not divide
any of $a(1), \ldots, a(N_0)$.
By (iii), there is some term divisible by $q$. Let $a(m)=tq$ be the first such term.
But now $tp^i < tq$ is a smaller candidate for $a(m)$, a contradiction.

(v) For any prime $p$ there is a term with $a(n)=p$. Again suppose not.
Using (i), choose $N_0$ large enough that $a(n)>p$ for all $n \ge N_0$.
By (iv), we can find an $n \ge N_0$ such that $a(n)=tp$ for some $t$.
Then $a(n+2)=p$, a contradiction.

(vi) All numbers appear. For if not, let $k$ be the smallest missing
number, and choose $N_0$ so that all of $1, \ldots, k-1$ have occurred in
$a(1), \dots, a(N_0)$.
Let $p$ be a prime dividing $k$. Since, by (iv),
$p$ divides infinitely many terms,
there is a number $N_1 > N_0$ such that $\gcd(a(N_1),k) > 1$.
This forces 
\beql{Eq1}
\gcd(a(n),k) > 1 \mbox{~for~\textbf{all}~}  n \ge N_1.
\eeq
(If not, there would be some $j \ge N_1$ where
$\gcd(a(j),k)>1$ and $\gcd(a(j+1),k)=1$, which would 
lead to $a(j+2)=k$.) But \eqn{Eq1} is impossible, because we know from (v) 
that infinitely many of the $a(n)$ are primes.
\end{proof}

\textbf{Remarks.}
The same argument, with appropriate modifications,
can be applied to many other sequences.
Let $\Omega$ be a sufficiently large set of 
positive integers,\footnote{For instance, let $\sP$ be an 
an infinite set of primes, and  take $\Omega$ to
consist of the positive numbers all of whose prime factors
belong to $\sP$. We could
also exclude any finite subset of numbers from $\Omega$.
We obtain the Yellowstone permutation by taking $\sP$ to be all primes, $\Omega$ to be the positive integers, and requiring that the sequence begin with $1,2,3$.}
and define a sequence $(c(n))_{n \ge 1}$ by specifying
that certain members of $\Omega$ must appear at the start
of the sequence (including $1$, if $1 \in \Omega$), and 
that thereafter $c(n) \in \Omega$ is the smallest
number not yet used which satisfies
$\gcd(c(n),c(n-2)) > 1$, $\gcd(c(n-1),c(n))=1$.
Then the resulting sequence will be a permutation of $\Omega$.
We omit the details.

For example, if we take $\Omega$ to be the odd positive integers,
and specify that the sequence begins $1,3,5$, we obtain 
\seqnum{A251413}.

Or, with $\Omega$ the positive integers again,
we can generalize our main sequence by taking 
the first three terms to be $1,x,y$ with $x>1$ and $y>1$ relatively prime.
For example, starting with $1,3,2$ gives \seqnum{A251555}, and
starting with $1,2,5$ gives \seqnum{A251554}, neither of which appears to merge 
with the main sequence,
whereas starting with $1,4,9$ merges with \seqnum{A098550} after five steps.
It would be interesting to know more about which of these sequences 
eventually merge.
It follows from Theorem~\ref{Th1} that a necessary and
sufficient condition for two sequences $(c(n))$,
$(d(n))$ of this type (that is, beginning $1,x,y$) to merge is that for some $m$, 
terms 1 through $m-2$ contain the same set of numbers,
and  $c(m-1)=d(m-1)$, $c(m)=d(m)$. 

\section{Growth of the sequence}\label{Sec3}

\begin{table}[!ht]
\caption{ The first 300 terms $a(20i+j), 0 \le i \le 14, 1 \le j \le 20$ of the Yellowstone permutation.
The primes (or downward spikes) are shown in red, the ``$\ka$p'' points (the upward spikes, or ``geysers'')  in blue.}
\label{Tab1}
{\scriptsize
$$
\begin{array}{r|rrrrrrrrrrrrrrrrrrrr}
i \backslash j& 1& 2& 3& 4& 5& 6& 7& 8& 9& 10& 11& 12& 13& 14& 15& 16& 17& 18& 19& 20  \\
\hline 
0 & 1&  {\color{red}2}&  {\color{red}3}&  {\color{blue}4}&  {\color{blue}9}&  8&  15&  14&  {\color{red}5}&  6&  {\color{blue}25}&  12&  35&  16&  {\color{red}7}&  10&  {\color{blue}21}&  20&  27&  22 \\
1& 39&  {\color{red}11}&  {\color{red}13}&  {\color{blue}33}&  {\color{blue}26}&  45&  28&  51&  32&  {\color{red}17}&  18&  {\color{blue}85}&  24&  55&  34&  65&  36&  91&  30&  49 \\
2& 38&  63&  {\color{red}19}&  42&  {\color{blue}95}&  44&  57&  40&  69&  50&  {\color{red}23}&  48&  {\color{blue}115}&  52&  75&  46&  81&  56&  87&  62 \\
3& {\color{red}29}&  {\color{red}31}&  {\color{blue}58}&  {\color{blue}93}&  64&  99&  68&  77&  54&  119&  60&  133&  66&  161&  72&  175&  74&  105&  {\color{red}37}&  70 \\
4& {\color{blue}111}&  76&  117&  80&  123&  86&  {\color{red}41}&  {\color{red}43}&  {\color{blue}82}&  {\color{blue}129}&  88&  135&  92&  125&  78&  145&  84&  155&  94&  165 \\
5& {\color{red}47}&  90&  {\color{blue}329}&  96&  203&  100&  147&  104&  141&  98&  153&  106&  171&  {\color{red}53}&  102&  {\color{blue}265}&  108&  185&  112&  195 \\
6& 116&  143&  114&  121&  118&  187&  {\color{red}59}&  110&  {\color{blue}177}&  122&  159&  {\color{red}61}&  120&  {\color{blue}427}&  124&  183&  128&  189&  130&  201 \\
7& 136&  {\color{red}67}&  126&  {\color{blue}335}&  132&  205&  134&  215&  138&  235&  142&  225&  {\color{red}71}&  140&  {\color{blue}213}&  146&  207&  {\color{red}73}&  144&  {\color{blue}365} \\
8& 148&  219&  152&  231&  158&  209&  {\color{red}79}&  154&  {\color{blue}237}&  160&  243&  164&  249&  170&  {\color{red}83}&  150&  {\color{blue}581}&  156&  217&  162 \\
9& 245&  166&  255&  172&  221&  168&  169&  174&  247&  176&  273&  178&  259&  {\color{red}89}&  182&  {\color{blue}267}&  184&  261&  188&  279 \\
10& 190&  291&  196&  {\color{red}97}&  180&  {\color{blue}679}&  186&  287&  192&  301&  194&  315&  202&  275&  {\color{red}101}&  198&  {\color{blue}505}&  204&  295&  206 \\
11& 285&  {\color{red}103}&  200&  {\color{blue}309}&  208&  297&  212&  253&  210&  299&  214&  325&  {\color{red}107}&  220&  {\color{blue}321}&  218&  303&  {\color{red}109}&  216&  {\color{blue}545} \\
12& 222&  305&  224&  345&  226&  327&  {\color{red}113}&  228&  {\color{blue}565}&  232&  339&  230&  333&  236&  351&  238&  363&  244&  319&  234 \\
13& 341&  240&  403&  242&  377&  246&  455&  248&  343&  250&  357&  254&  289&  {\color{red}127}&  272&  {\color{blue}381}&  256&  369&  260&  387 \\
14& 262&  375&  {\color{red}131}&  252&  {\color{blue}655}&  258&  355&  264&  395&  266&  405&  268&  385&  274&  371&  {\color{red}137}&  280&  {\color{blue}411}&  278&  393 
\end{array}
$$
}
\end{table}

From studying the first 100 million terms of the sequence
$(a(n))$, we believe we have an accurate model of how
the sequence grows.
However, at present we have no proofs for any of the following statements.
They are merely empirical observations.

The first 212 terms are exceptional (see Table \ref{Tab1}).
Starting at the 213th term, it appears that the sequence
is governed by what we shall call:

\textbf{Hypothesis A.}
(``A'' stands for ``alternating''.)
The sequence alternates between even and odd composite terms, except that, when an even term is reached which is twice a prime, the alternation of even and odd terms 
is disrupted, and we see five successive terms of the form
\beql{Eq2}
2p,~ 2i+1, ~ p, ~ 2j, ~ \ka p,
\eeq
where $p$ is an odd prime,  $i$ and $j$ are integers,
and $\ka < p$ is the least odd prime that does not divide~$j$.
The $\ka p$ terms are the ``geysers'' (\seqnum{A251544}).

For example, terms $a(213)$ to $a(217)$ are
$$
202=2 \cdot 101,~ 275,~ 101, ~ 198 = 2\cdot 3^2\cdot 11, ~ 505=5\cdot 101\,.
$$
Hypothesis A is only a conjecture, since we cannot rule out the possibility that 
this behavior breaks down at some much later point in the sequence.
It is theoretically possible, for example, that a term that is twice a prime is not
followed two steps later by the prime itself (as happens  after
$a(8)=14$, which is followed two steps later by $a(10)=6$ rather than $7$).
However, as we shall argue later in this section, this is unlikely to happen.

Under Hypothesis A, most of the time
the sequence alternates between even and odd composite terms,
and the $n$th term $a(n)$ is about $n$, to
a first approximation. However, the primes appear later than they
should, because $p$ cannot appear until the sequence
first reaches $2p$, which takes about $2p$ steps,
and so the primes are roughly on the line
$f(x)=x/2$.  On the other hand,
the term $\ka p$ in \eqn{Eq2} appears earlier than it should,
and lies roughly on the line $f(x)=\ka x/2$.

Continuing to assume that Hypothesis A holds for terms 213 onwards, 
we can give heuristic arguments that lead to better
asymptotic estimates, as follows.
Guided by \eqn{Eq2}, we divide the terms of the sequence
into several types: type $E$ terms, consisting of all the even terms;
type $p$, all the odd primes;
types $\ka p$ for $\ka= 3,5,7,11, \ldots$, the terms that appear
two steps after a prime; and
type $C$, all the odd composite terms that are not of type $\ka p$ for any $\ka$.

From term 213 onwards, even and odd terms (more precisely, types $E$ and $C$)
 alternate, except when the even term is 
twice a prime, when we see the five-term subsequence \eqn{Eq2},
containing two $E$ terms, one $p$ term,
one $\ka p$ term for some odd prime $\ka$,  and one $C$ term.
Between terms 213 and $n$, we will see about $\la$ of these five-term subsequences,
where $\la$ is the number of terms in that range that
are twice a prime. $\la$ is therefore approximately\footnote{We assume $n$ is very large,
and  ignore the fact that the first 212 terms are slightly exceptional---asymptotically 
this makes no difference.}
$\pi(a(n)/2)$, where $\pi(x)$ is the number
of primes $\le x$.

There are $n-5\la$ terms not in the 5-term subsequences,
so the total number of even terms out of $a(1), \ldots, a(n)$ is roughly
\beql{Eq3a}
\frac{n-5\la}{2} ~+~ 2\la ~=~ \frac{n-\la}{2} ~ \cong ~ \frac{n-\pi(\frac{a(n)}{2})}{2},
\eeq
where $\cong$ signifies ``is approximately equal to''.
Although the even terms do not increase monotonically
(compare Table~\ref{Tab1}), it appears to be a good approximation
to assume that, on the average, each even term
contributes 2 to the growth of the even subsequence,
and so, if $a(n)$ is an even term, we obtain
\beql{Eq3b}
a(n) ~\cong~ n-\pi\left(\frac{a(n)}{2}\right).
\eeq
In other words, the even terms should lie on or close to the 
curve $y=f_E(x)$ defined by the functional equation
\beql{Eq3c}
y~+~ \pi\left(\frac{y}{2}\right) ~=~ x.
\eeq
The primes then lie on the curve $f_p(x) = \frac{1}{2}f_E(x)$,
and the $\ka p$ terms on the curve
$f_{\ka p}(x) = \frac{\ka}{2} f_E(x)$ 
for $\ka = 3,5,7, \ldots$.

Although the reasoning that led us to \eqn{Eq3c} was far from rigorous,
it turns out that \eqn{Eq3c} is a remarkably good fit to the graph of
the even terms, at least for the first  $10^8$ terms.
We  solved \eqn{Eq3c} numerically, and computed the residual errors $a(n) - f_E(n)$.
The fit is very good indeed for the ``normal'' even terms,
those that do not belong to  the 5-term subsequences.
As can be seen from
Fig.~\ref{Fig2En}, up to $n=10^7$, the maximum error is less than 40, in
numbers which are around $10^7$.

\begin{figure}[!ht]
\centerline{\includegraphics[angle=0, width=4.00in]{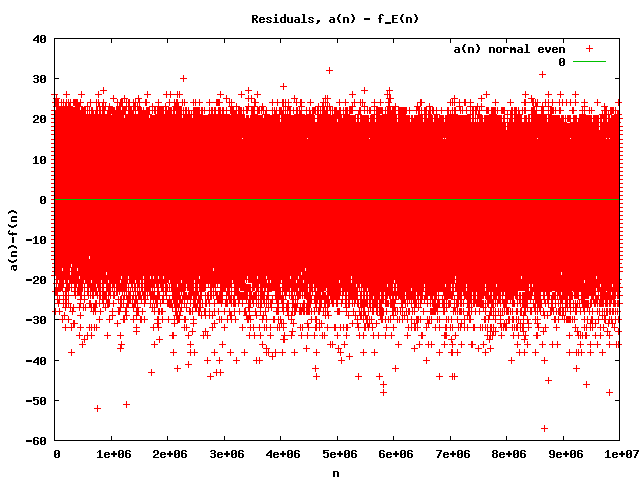}}
\caption{The difference between the ``normal'' even terms and the
approximation $f_E(n)$ given in \eqn{Eq3c} is
at most 40 for $n \le 10^7$.}
\label{Fig2En}
\end{figure}

\begin{figure}[!ht]
\centerline{\includegraphics[angle=0, width=4.00in]{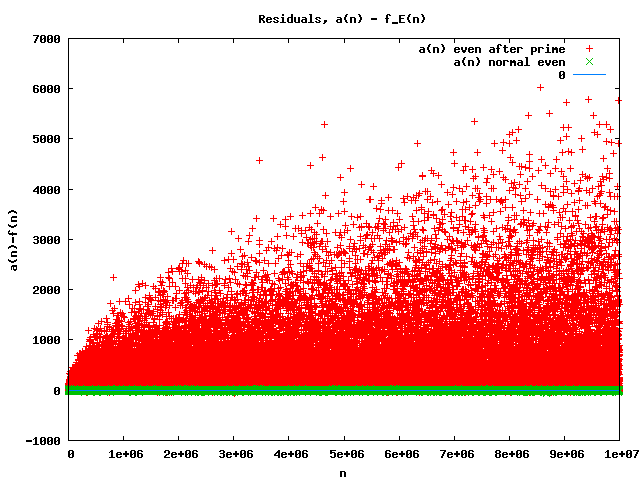}}
\caption{The difference between the even terms that follow a prime and the
approximation $f_E(n)$ given in \eqn{Eq3c}  is
at most 6000 for $n \le 10^7$.
The green points are the same errors shown in the previous figure, plotted
on this scale.}
\label{Fig2Ea}
\end{figure}

The fit is still good for the even terms  in the five-term subsequences,
 although not so remarkable,
as can be seen in Fig.~\ref{Fig2Ea}.
Up to $n=10^7$ there are errors as large as 6000, which
is on the order of $\sqrt{n}$. The errors for the 
``normal'' even terms are shown in this figure in green.

If we use $\pi(x) \sim  x/\log x$ in \eqn{Eq3c}, we obtain
\beql{Eq3d}
f_E(x) ~= ~ x \, \left( 1 ~-~ \frac{1}{2 \log x} ~+~ o\left(\frac{1}{\log x}\right) \right).
\eeq
However, \eqn{Eq3c} is a much better fit than just using the first two
terms on the right side of \eqn{Eq3d}.

We can study the curve $f_C(x)$ containing the type $C$ terms
(the odd composite terms not of type $\ka p$) in a similar manner.
This is complicated by the fact that the values of $\ka$ are hard to predict.
We therefore use a probabilistic model, and let $\si(\ka)$ denote
the probability that the multiplier in a $\ka p$ term is $\ka$.
Empirically, $\si(3) \cong 0.334$,
$\si(5) \cong 0.451$,
$\si(7) \cong 0.174, \ldots$.
The number of type $C$ terms in the first $n$ terms is (compare \eqn{Eq3a})
\beql{Eq4a}
\frac{n-5\la}{2} ~+~ \la ~=~ \frac{n-3\la}{2} ~ \cong ~ \frac{n-3\pi(\frac{f_E(n)}{2})}{2}.
\eeq
However, type $C$ terms skip over the primes, and we expect
to see $\pi(f_C(n))$ primes $\le f_C(n)$.
Type $C$ terms also skip over the $\ka p$ terms that have
already appeared in the sequence.
For a given value of  $\ka$, terms of type $\ka p$
will have been skipped over if $\ka p \le f_C(n)$, and
if that value of $\ka$ was chosen, so the number of $\ka p$ terms
we skip over is
$$
\sum_{ {  \mbox{~odd~primes~}} \ka} \si(\ka) \,\pi \left( \frac{f_C(n)}{\ka} \right),
$$
where here and in the next two displayed equations the summation
ranges over all odd primes $\ka \le \sqrt{f_C(n)}$.
Each of these events contributes 2, on the average,
to the growth of the $C$ terms, so we obtain
\beql{Eq4b}
f_C(n) ~\cong~ n-3\pi\left(\frac{f_E(n)}{2}\right)+2\pi(f_C(n))
+2 \sum_{ {  \mbox{~odd~primes~}} \ka} \si(\ka) \,\pi \left( \frac{f_C(n)}{\ka} \right).
\eeq
In other words, the type $C$ terms should lie on or close to the curve $y = f_C(x)$
defined by the functional equation
\beql{Eq4c}
y ~-~ 2 \pi(y) ~-~
2 \sum_{ {  \mbox{~odd~primes~}} \ka} \si(\ka) \,\pi \left( \frac{y}{\ka} \right)
 ~=~ x ~-~ 3\pi\left(\frac{f_E(x)}{2}\right).
\eeq

Equation \eqn{Eq4c} can be solved
numerically, using the values of $f_E(x)$ computed from \eqn{Eq3c},
and gives a  good fit to the graph of the type $C$ terms.
As can be seen from Fig.~\ref{Fig2C}, the errors in the first $10^7$ terms
are on the order of $5 \sqrt{n}$. It is not surprising
that the errors are larger for type $C$ terms than type $E$ terms,
since as can be seen in Fig.~\ref{Fig2}, the curve
with the $C$ points is much thicker than the $E$ curve.

\begin{figure}[!ht]
\centerline{\includegraphics[angle=0, width=4.00in]{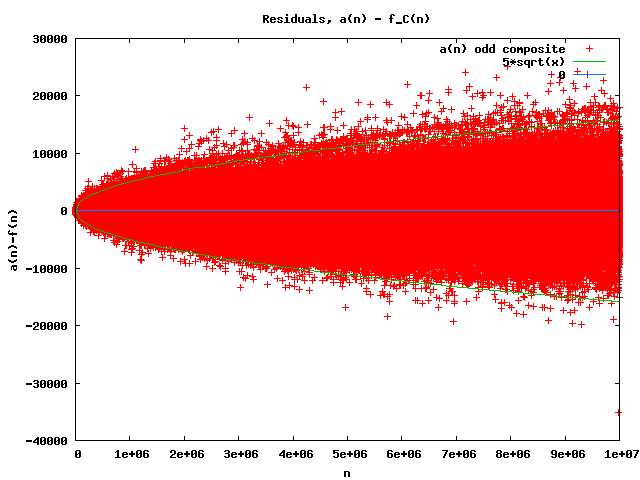}}
\caption{The difference between the odd composite (or type $C$) terms and the
approximation $f_C(n)$  defined implicitly by \eqn{Eq4c} is
on the order of $5 \sqrt{n}$  for $n \le 10^7$.}
\label{Fig2C}
\end{figure}

If we use $\pi(x) \sim  x/\log x$ in \eqn{Eq4c}, we obtain
\beql{Eq4d}
f_C(x) ~= ~ x \, \left( 1 ~+~ \frac{\al}{\log x} ~+~ o\left(\frac{1}{\log x}\right) \right),
\eeq
where 
\beql{Eq4e}
\al ~=~ \frac{1}{2} ~+~ 
2 \sum_{ {  \mbox{~odd~primes~}} \ka \ge 3} \frac{ \si(\ka)}{\ka}
  ~\cong~ 0.96,
\eeq
and now the summation is over all odd primes $\ka$.
 
To summarize, our estimates for the curves $f_E(x)$ and $f_C(x)$
containing the terms of types $E$ and $C$ are given by
\eqn{Eq3c} and \eqn{Eq4c}.
Equations \eqn{Eq3d} and \eqn{Eq4d} have
a simpler form but are less precise.
The primes  lie on the curve $f_p(x) = \frac{1}{2}f_E(x)$,
and the $\ka p$ terms on the curves
$f_{\ka p}(x) = \frac{\ka}{2} f_E(x)$ 
for $\ka = 3,5,7, \ldots$.
In Fig.~\ref{Fig2},
reading counterclockwise from the horizontal axis,
we see the curves 
$f_p(x)$, $f_E(x)$,  the red line $f(x)=x$, then $f_C(x)$,
$f_{3p}(x)$, $f_{5p}(x)$, $f_{7p}(x)$, $f_{11p}(x)$,
and a few points from $f_{\ka p}(x)$ for $\ka \ge 13$.
At this scale, the curves look straight.

To see why Hypothesis A is unlikely to fail,
note that when
we add an even number to the sequence, most
of the time it belongs to the interval $[m_E,M_E]$ (\seqnum{A251546},  \seqnum{A251557}),
which we call the {\em even frontier},
where $m_E$ is the smallest even number that is not yet in the sequence, and 
$M_E$ is 2 more than the largest even number that has appeared.
The {\em odd composite frontier} $[m_C, M_C]$ (\seqnum{A251558}, \seqnum{A251559})
 is defined similarly for the type $C$ points.
For example, when $n=10^6$, $a(10^6)=1094537$,
the even frontier is $[960004,\ldots,960234]$
and the odd composite frontier is $[1092467,\ldots, 1097887]$.
In fact, at this point, no even number in the range
$960004,\ldots,960230$ is in the sequence.
What we see here is typical of the general situation:
the length of the even frontier,  $M_E-m_E$, is much less
 than the length of the odd composite frontier, $M_C-m_c$;
most of the terms in the even frontier are available; and the two
frontiers are well separated.
As long as this continues, the even and odd frontiers
will remain separated, and Hypothesis A will hold.
The much larger width of the odd composite frontier is reflected in the greater thickness of  the``C'' curve in Figure~\ref{Fig2}.

We know from the proof of Theorem \ref{Th1}
that if $p<q$ are primes, the first term divisible by $p$ occurs
before the first term divisible by $q$.
But we do not know that $p$ itself occurs before $q$.
This would be a consequence of Hypothesis A, but perhaps it
can be proved by arguments similar to those used to prove Theorem \ref{Th1}.
Sequences \seqnum{A252837} and \seqnum{A252838} contain
additional information related to Hypothesis A.

The OEIS contains a number of other sequences
(e.g.,
\seqnum{A098548},
\seqnum{A249167},
\seqnum{A251604},
\seqnum{A251756},
\seqnum{A252868})
whose definition has a similar flavor to that of the Yellowstone permutation.
Two sequences contributed by Adams-Watters are 
especially noteworthy: \seqnum{A252865} is an
analog of \seqnum{A098550} for square-free numbers, and
\seqnum{A252867} is a set-theoretic version.

\begin{figure}[!ht]
\centerline{\includegraphics[angle=0, width=6.75in]{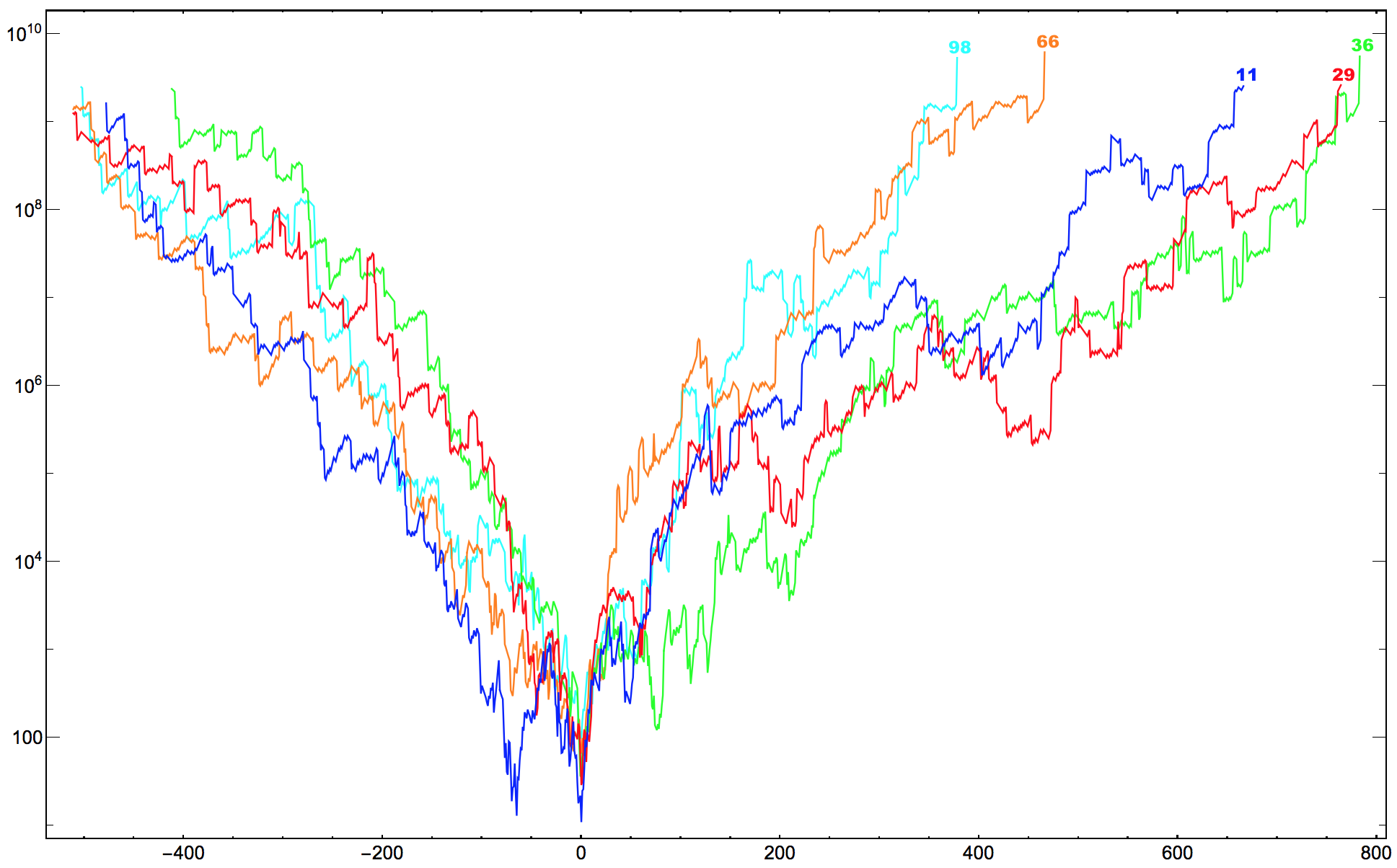}}
\caption{Portions of the conjecturally infinite
orbits whose smallest terms are
$11$ (blue), $29$ (red),
$36$ (green),  $66$ (orange),  $98$ (cyan).}
\label{Fig3}
\end{figure}

\section{Orbits under the permutation}\label{Sec4}

Since the sequence is a permutation of the positive integers, 
it is natural to study its orbits.
It appears that the only fixed points are 
$1, 2, 3, 4, 12, 50, 86$
(\seqnum{A251411}). 
There are certainly no other fixed points below $10^9$, and
Fig.~\ref{Fig2} makes it very plausible that there are
no further points on the red line.

At present, 27 finite cycles are known besides the seven fixed points.
For example, $6$ is in the cycle $(6,8,14,16,10)$.
The finite cycle with the largest minimum term known to
date is the cycle of length 45 containing 756023506.

We conjecture that, on the other hand,
 almost all positive numbers belong to infinite orbits. 
See Fig.~\ref{Fig3} for portions of the conjecturally infinite
orbits whose smallest terms are respectively
$11$ (\seqnum{A251412}), $29$, $36$, $66$, and $98$
(cf.\ \seqnum{A251556}).
The orbits have been displaced sideways so that the conjectured  minimal value is
positioned at $x=0$.
For example, the blue curve to
the right of $x=0$ shows the 
initial portion of the trajectory of $11$ under
repeated applications of the Yellowstone
permutation, while the curve to the left
 shows the trajectory under repeated applications of
the inverse permutation.
In other words, the blue curve, from upper left to upper right,
is a section of the orbit whose minimal value appears to be 11.
The inverse trajectory of $11$ has near-misses after three steps, 
when it reaches 18, and after 70 steps, when it reaches 19,
but once the numbers get large it seems that there
is little chance  that   the forward and inverse trajectories will ever meet,
implying that the orbit is infinite.Ä

However, because of the erratic appearance of the trajectories
in Fig.~\ref{Fig3}, there is perhaps a greater possibility that these paths may
eventually close, or merge, compared with the situation
for other well-known permutations. For example,
in the case of the ``amusical permutation'' 
of the nonnegative integers (\seqnum{A006368}) studied by
Conway and Guy \cite{JHC1972, JHC2013, UPINT}, the
empirical evidence that most orbits are infinite
is much stronger---compare Figs.~1 and 2 of \cite{JHC2013}
with our Fig.~\ref{Fig3}. 

\section{Acknowledgments}
We would like to thank Bradley Klee for the key idea of considering
terms of the form $p^i$ in
part (iv) of the proof of Theorem \ref{Th1}.

Sequence \seqnum{A098550} was the subject of many discussions
on the {\em Sequence Fans Mailing List} in 2014.
Franklin~T.~Adams-Watters observed that when the first 10000 terms
are plotted, the slopes of the various lines in the graph 
were, surprisingly,  not recognizable as rational numbers.
Jon E. Schoenfield noticed that if the third (or ``$C$'') line  was ignored,
the ratios of the other slopes were consecutive primes.
See \S\ref{Sec3} for our conjectured explanation
of these observations.
L.~Edson~Jeffery observed that there appear to be only seven fixed points
(see \S\ref{Sec4}).

Other contributors to the discussion include
Maximilian~F.~Hasler,
Robert~Israel,
Beno\^{\i}t~Jubin,
Jon~Perry,
Robert~G.~Wilson~v,
and Chai Wah~Wu.

We thank several of these people, especially Beno\^{\i}t~Jubin and
Jon E. Schoenfield,
for sending comments on the manuscript.

\bigskip
\hrule
\bigskip

\noindent 2010 {\it Mathematics Subject Classification}:
Primary 11Bxx, 11B83, 11B75.

\noindent \emph{Keywords: } Number sequence,
EKG sequence, permutation of natural numbers.

\bigskip
\hrule
\bigskip

\noindent (Concerned with sequences
\seqnum{A006368},
\seqnum{A064413},
\seqnum{A098548},
\seqnum{A098550},
\seqnum{A249167},
\seqnum{A249943},
\seqnum{A251237},
\seqnum{A251411}--\seqnum{A251413}, 
\seqnum{A251542}--\seqnum{A251547},
\seqnum{A251554}--\seqnum{A251559},
\seqnum{A251604},
\seqnum{A251621},
\seqnum{A251756},
\seqnum{A252837},
\seqnum{A252838},
\seqnum{A252865},
\seqnum{A252867},
\seqnum{A252868},
\seqnum{A253048}, 
\seqnum{A253049}.
The OEIS entry \seqnum{A098550} contains
cross-references to many other related sequences.)

\bigskip
\hrule
\bigskip


\begin{thebibliography}{99}

\bibitem{JHC1972}
J. H. Conway,
Unpredictable iterations, 
{\em Proceedings of the Number Theory Conference (Univ. Colorado, Boulder, CO, 1972)},
pp.~49--52.

\bibitem{JHC2013}
J. H. Conway,
On unsettleable arithmetical problems,
{\em Amer. Math. Monthly}, 
{\textbf 120} (2013), 192-198.

\bibitem{Erdos}
P. Erd\H{o}s, R. Freud, and N. Hegyv\'{a}ri, 
 Arithmetical properties of permutations of integers,
 {\em Acta Math. Hungar.}, {\textbf 41} (1983),  169--176.

\bibitem{UPINT}
R. K. Guy,
{\em Unsolved Problems in Number Theory}, 
Springer, 2nd. ed., 1994. See \S E17.

\bibitem{EKG}
J. C. Lagarias, E. M. Rains, and N. J. A. Sloane,
The EKG Sequence,
{\em Experimental Math.}, 
{\textbf 11} (2002), 437--446.

\bibitem{OEIS}
The OEIS Foundation Inc.,
{\em The On-Line Encyclopedia of Integer Sequences},
\url{http://oeis.org}.

\end{thebibliography}
\end{document}